\documentclass[10pt,a4paper]{amsart}
\usepackage[cp1250]{inputenc}
\usepackage{amsmath}
\usepackage{amsfonts}
\usepackage{amssymb}

\usepackage{amsthm}

\newtheorem{theorem}{Theorem}[section]
\newtheorem{lemma}[theorem]{Lemma}
\newtheorem{proposition}[theorem]{Proposition}
\newtheorem{corollary}[theorem]{Corollary}

\theoremstyle{definition}
\newtheorem{definition}[theorem]{Definition}
\newtheorem{remark}[theorem]{Remark}
\newtheorem*{acknowledgment}{Acknowledgment}

\newcommand{\R}{\mathbb{R}}
\newcommand{\N}{\mathbb{N}}
\newcommand{\M}{\mathcal{M}}
\newcommand{\rec}{\operatorname{rec}}
\newcommand{\spann}{\operatorname{span}}
\newcommand{\loc}{\operatorname{loc}}

\begin{document}

\title{Functions on a convex set which are both $ \omega $-semiconvex and $ \omega $-semiconcave II}
\author{Václav Kryštof}

\address{Charles University, Faculty of Mathematics and Physics, Sokolovská 83, 186 75 Praha 8, Karlí­n, Czech Republic}

\email{vaaclav.krystof@gmail.com}



\begin{abstract}
In a recent article (2022) we proved with L. Zajíček that if $ G\subset\R^n $ is an unbounded open convex set that does not contain a translation of a convex cone with non-empty interior, then there exist $ f:G\to\R $ and a concave modulus $ \omega $ such that $ \lim_{t\to\infty}\omega(t)=\infty $, $ f $ is both semiconvex and semiconcave with modulus $ \omega $ and $ f\notin C^{1,\omega}(G) $. Here we improve the previous result as follows: If $ G $ is as above and $ \omega(t)=t^{\alpha} $ for some $ \alpha\in(0,1) $, then there exists $ f:G\to\R $ that is both semiconvex and semiconcave with modulus $ \omega $ and $ f\notin C^{1,\alpha}(G) $. This result has immediate consequences concerning a first-order quantitative converse Taylor theorem and the problem whether $ f\in C^{1,\alpha}(G) $ whenever $ f $ is smooth in a corresponding sense on all lines.
\end{abstract}

\maketitle

\section{Introduction}

If $ G\subset\R^n $ is an open convex set, then $ f:G\to\R $ is said to be (classically) semiconvex if there exists $ C\geq 0 $ such that the function
\begin{align*}
x\mapsto f(x)+C\vert x\vert^2,\quad x\in G,
\end{align*}
is convex. We say that $ f:G\to\R $ is (classically) semiconcave if $ -f $ is semiconvex. Then it is well-known (see \cite[Section 3]{KZ} for references) that
\begin{align}\label{eq_classic}
f:G\to\R\;\textnormal{is both semiconvex and semiconcave}\;\Longleftrightarrow f\in C^{1,1}(G).
\end{align}

The notion of semiconvex, resp. semiconcave, functions is generalized to the notion of funtions that are semiconvex, resp. semiconcave, with modulus $ \omega $ (see \cite[Definition 2.1.1]{CS}). We will call these functions $ \omega $-semiconvex, resp. $ \omega $-semiconcave, for short (see Definition \ref{de_modulus} and Definition \ref{de_semiconvex} below). Note that a function is semiconvex, resp. semiconcave, if and only if it is $ \omega $-semiconvex, resp. $ \omega $-semiconcave, with $ \omega(t)=Ct $ for some $ C\geq 0 $ (see \cite[Proposition 1.1.3]{CS}). For the theory and application of $ \omega $-semiconvex, resp. $ \omega $-semiconcave, functions, see the monograph \cite{CS}.

This article deals with functions that are both $ \omega $-semiconvex and $ \omega $-semiconcave. Let $ G\subset\R^n $ be an open convex set and $ \omega $ a modulus. It is easy to show that if $ f\in C^{1,\omega}(G) $, then there exists $ C>0 $ such that $ f $ is both $ (C\omega) $-semiconvex and $ (C\omega) $-semiconcave (see \cite[Proposition 2.1.2]{CS}). So we will consider the natural problem whether also the opposite implication holds. This implication can be equivalently formulated as follows:
\begin{align}\label{eq_basic}
f:G\to\R\;\textnormal{is both}\;\omega\textnormal{-semiconvex and}\;\omega\textnormal{-semiconcave}\Longrightarrow f\in C^{1,\omega}(G).
\end{align}

Recall (see \eqref{eq_classic}) that implication \eqref{eq_basic} holds if $ \omega(t)=Ct $ for some $ C\geq 0 $. Further, by Remark \ref{re_continuous} (i) below and \cite[p. 839]{KZ}, implication \eqref{eq_basic} holds if either $ G $ is bounded or contains a translation of a convex cone with non-empty interior. However, \eqref{eq_basic} doesn't hold in general. Namely, we proved (see \cite[Theorem 1.9]{KZ}) that if $ G\subset\R^n $ is an unbounded open convex set which does not contain a translation of a convex cone with non-empty interior, then there exists a concave modulus $ \omega $ such that $ \lim_{t\to\infty}\omega(t)=\infty $ and \eqref{eq_basic} doesn't hold. The main disadvantage of \cite[Theorem 1.9]{KZ} is that $ \omega $ depends on $ G $ and may be irregular. In this article we improve \cite[Theorem 1.9]{KZ} in the following way:

\begin{theorem}\label{th_1}
Let $ G\subset\R^n $ ($ n\geq 2 $) be an unbounded open convex set that doesn't contain a translation of a cone with non-empty interior. Let $ \alpha\in (0,1) $ and set $ \omega(t):=t^\alpha $ for all $ t\geq 0 $. Then there exists $ f:G\to\R $ that is both $ \omega $-semiconvex and $ \omega $-semiconcave and $ f\notin C^{1,\alpha}(G) $.
\end{theorem}

So Theorem \ref{th_1} gives a positive answer to \cite[Question 6.10 (3), (4)]{KZ}. By \eqref{eq_condition}, Theorem \ref{th_1} is a special case of Theorem \ref{th_main} below which works with more general moduli (namely those with property \eqref{de_condition} below).

The strategy of the proof of Theorem \ref{th_main} is similar to that used in \cite{KZ} where we reduced (using the notion of recessive cones) the case of an open convex subset of $ \R^n $ to the case of an open convex subset of $ \R^2 $ of the special form $ G_{\eta}:=\lbrace(x,y)\in\R^2:x>0,\;\vert y\vert<\eta(x)\rbrace $ where $ \eta:(0,\infty)\to(0,\infty) $ is concave and satisfies $ \lim_{t\to\infty}\eta(t)/t=0 $. However, the construction of $ f:G_{\eta}\to\R $ in the proof of Proposition \ref{pr_main} is much more sophisticated than the corresponding construction used in \cite{KZ}. The present construction is based on the notion of modulus $ \omega_{\eta} $ (see Definition \ref{de_omegaeta} below and for the motivation Remark \ref{re_f} below).

It's worth mentioning that if $ \omega $ is a modulus, then functions that are both $ (C\omega) $-semiconvex and $ (C\omega) $-semiconcave for some $ C\geq 0 $ can be characterized via ``Taylor approximation of the first order" or via ``$ C^{1,\omega} $-smoothness on all lines". See Lemma \ref{le_lines} below for a precise formulation. Thus Theorem \ref{th_1} can be reformulated as follows:

\begin{corollary}\label{co_1}
Let $ G\subset\R^n $ ($ n\geq 2 $) be an unbounded open convex set that doesn't contain a translation of a cone with non-empty interior. Let $ \alpha\in (0,1) $ and set $ \omega(t):=t^\alpha $ for all $ t\geq 0 $. Then there exists $ f:G\to\R $ that is Fréchet differentiable,
\begin{itemize}
\item[(i)] for every $ a\in G $ and $ v\in\R^n $, $ \vert v\vert=1 $, the derivative of the function $ t\mapsto f(a+tv) $ is uniformly continuous on $ \lbrace t\in\R:a+tv\in G\rbrace $ with modulus $ \omega $,
\item[(ii)] $ \vert f(x+h)-f(x)-f'(x)[h]\vert\leq\vert h\vert^{1+\alpha} $ for all $ x,x+h\in G $ and
\item[(iii)] $ f\notin C^{1,\alpha}(G) $.
\end{itemize}
\end{corollary}

Let us remark that Corollary \ref{co_1} improves \cite[Theorem 6.8]{KZ}. Note also that Theorem \ref{th_main} can be reformulated (using Lemma \ref{le_lines}) in the same way as Theorem \ref{th_1} and so we can get a more general statement than Corollary \ref{co_1}.

\section{Preliminaries}

Throughout this article, the norm on $ \R^n $ is always Euclidean. For $ A\subset\R^n $ the symbol $ \spann(A) $ denotes the linear span of $ A $. By a cone in $ \R^n $ we mean a set $ C\subset\R^n $ such that $ \lambda x\in C $ for each $ x\in C $ and $ \lambda>0 $. If $ G\subset\R^n $ is an open set, $ f:G\to\R $, $ x\in G $ and $ h\in\R^n $, then $ f'(x) $ denotes the Fréchet derivative of $ f $ at $ x $ and $ f'(x)[h] $ denotes its evaluation at $ h $.

\begin{definition}\label{de_modulus}
We denote by $ \M $ the set of all $ \omega:[0,\infty)\to[0,\infty) $ which are non-decreasing and satisfy $ \lim_{h\to0+}\omega(h)=0 $. The members of $ \M $ are called moduli.
\end{definition}

\begin{definition}
Let $ G\subset\R^n $ be an open set, $ \omega\in\M $ and $ \alpha\in (0,1] $. Then we denote by $ C^{1,\omega}(G) $ the set of all Fréchet differentiable $ f:G\to\R $ such that $ f' $ is uniformly continuous with modulus $ C\omega $ for some $ C>0 $. If $ \omega(t)=t^{\alpha} $ for every $ t\geq 0 $, then we also write $ C^{1,\alpha}(G) $ instead of $ C^{1,\omega}(G) $.
\end{definition}

\begin{definition}\label{de_semiconvex}
Let $ G\subset\R^n $ be an open convex set and $ \omega\in\M $.
\begin{itemize}
\item We say that $ f:G\to\R $ is semiconvex with modulus $ \omega $ (or $ \omega $-semiconvex for short) if
\begin{equation*}
f(\lambda x+(1-\lambda)y)\leq\lambda f(x)+(1-\lambda)f(y)+\lambda(1-\lambda)\vert x-y\vert\omega(\vert x-y\vert)
\end{equation*}
for every $ x,y\in G $ and $ \lambda\in[0,1] $.
\item We say that $ f:G\to\R $ is semiconcave with modulus $ \omega $ (or $ \omega $-semiconcave) if $ -f $ is semiconvex with modulus $ \omega $.
\end{itemize}
\end{definition}

\begin{remark}\label{re_continuous}
Let $ G\subset\R^n $ be an open convex set, $ f:G\to\R $ and $ \omega\in\M $. Then the following hold (see \cite[Theorem 2.1.7]{CS} and \cite[Theorem 3.3.7]{CS}):
\begin{itemize}
\item[(i)] If $ f $ is $ \omega $-semiconvex or $ \omega $-semiconcave, then $ f $ is continuous.
\item[(ii)] If $ f $ is both $ \omega $-semiconvex and $ \omega $-semiconcave, then $ f\in C^1(G) $.
\end{itemize}
\end{remark}

As already mentioned in the introduction, we have the following:

\begin{lemma}\label{le_lines}
Let $ G\subset\R^n $ be an open convex set, $ f:G\to\R $ and $ \omega\in\M $. Then the following are equivalent:
\begin{itemize}
\item[(i)] There exists $ C_1>0 $ such that $ f $ is both $ (C_1\omega) $-semiconvex and $ (C_1\omega) $-semiconcave.
\item[(ii)] There exists $ C_2>0 $ such that for every $ a\in G $ and $ v\in\R^n $, $ \vert v\vert=1 $, the function $ t\mapsto f(a+tv) $ is differentiable on $ \lbrace t\in\R:a+tv\in G\rbrace $ and its derivative is uniformly continuous with modulus $ C_2\omega $.
\item[(iii)] $ f $ is Fréchet differentiable and there exists $ C_3>0 $ such that
\begin{align*}
\vert f(x+h)-f(x)-f'(x)[h]\vert\leq C_3\vert h\vert\omega(\vert h\vert),\quad x,x+h\in G.
\end{align*}
\end{itemize}
\end{lemma}

\begin{proof}
It follows from Remark \ref{re_continuous} (i) and \cite[Proposition 6.4 (i), (iv)]{KZ} that (i) implies (ii) and (iii). Further, it follows easily from \cite[Lemma 2.3 (i)]{K} and \cite[Proposition 2.4 (ii)]{K} that (ii) implies (i). Finally, (iii) implies (i) by \cite[Proposition 6.4 (ii)]{KZ}.
\end{proof}

Recall that $ \omega\in\M $ is called sub-additive if $ \omega(x+h)\leq\omega(x)+\omega(h) $ for all $ x,h\geq 0 $.

\begin{remark}\label{re_additive}
Let $ \omega\in\M $. Then the following hold:
\begin{itemize}
\item[(i)] If $ \omega $ is concave, then it is sub-additive.
\item[(ii)] If $ \omega $ is sub-additive, then there exists a concave $ \varphi\in\M $ such that $ \omega\leq\varphi\leq 2\omega $.
\end{itemize}
Both of these statements are well-known. The proof of (i) is easy. Statement (ii) is due to Stechkin (see \cite[Lemma 2.3]{DZ} for references).
\end{remark}

\begin{remark}\label{re_concave}
Let $ f:(0,\infty)\to [0,\infty) $ be concave. Then the following hold:
\begin{itemize}
\item[(i)] $ f $ is non-decreasing.
\item[(ii)] The function $ t\mapsto f(t)/t $, $ t>0 $, is non-increasing.
\item[(iii)] For every $ t>0 $ and $ c\geq 1 $ we have $ f(ct)\leq cf(t) $.
\end{itemize}
All these facts are easy and well-known. Note that (iii) is just a reformulation of (ii). However, a useful one.
\end{remark}

As in \cite{KZ} we will use the following auxiliary notation:

\begin{definition}\label{de_gn}
Let $ n\in\N $ and $ \omega\in\M $. Then we denote by $ \mathcal{G}_n(\omega) $ the set of all open convex sets $ G\subset\R^n $ for which there exist $ a\in G $, $ x\in\R^n\setminus\lbrace 0\rbrace $ and $ f\in C^1(G) $ such that $ f $ is both $ \omega $-semiconvex and $ \omega $-semiconcave, $ M:=\lbrace a+\lambda x:\lambda\geq 0\rbrace\subset G $ and $ f' $ is uniformly continuous on $ M $ with modulus $ C\omega $ for no $ C>0 $.
\end{definition}

\begin{remark}\label{re_gn}
Let $ n\in\N $ and $ \omega\in\M $. Then:
\begin{itemize}
\item[(i)] By Remark \ref{re_continuous} (ii) we can equivalently write $ f:G\to\R $ instead of $ f\in C^1(G) $ in the definition of $ \mathcal{G}_n(\omega) $.
\item[(ii)] If $ G\in\mathcal{G}_n(\omega) $, then \eqref{eq_basic} clearly doesn't hold.
\item[(iii)] If $ b\in\R^n $ and $ G+b\in\mathcal{G}_n(\omega) $, then it is easy to check that $ G\in\mathcal{G}_n(\omega) $.
\end{itemize}
\end{remark}

For the proof of the following lemma see \cite[Example 5.1]{KZ}.

\begin{lemma}\label{le_kz1}
Let $ \omega\in\M $ satisfy
\begin{align*}
\liminf_{h\to 0+}\frac{\omega(h)}{h}>0,\;\lim_{h\to\infty}\frac{\omega(h)}{h}=0.
\end{align*}
Then $ \R\times (0,1)\in\mathcal{G}_2(\omega) $.
\end{lemma}

As in \cite{KZ}, we will work with recession cones of unbounded convex sets $  A\subset\R^n $ and use the notation $ \rec(A) $ (instead of $ 0^+A $ used in \cite{R}).

\begin{definition}
Let $ \emptyset\neq A\subset\R^n $ be a convex set. Then we set
\begin{align*}
\rec(A):=\lbrace x\in\R^n:(\forall a\in A)(\forall\lambda\geq 0)\;a+\lambda x\in A\rbrace.
\end{align*}
\end{definition}

Now we recall some basic properties of recession cones. For the proof see \cite[Lemma 2.10]{KZ}.

\begin{lemma}\label{le_rec}
Let $ n\in\N $ and let $ \emptyset\neq G\subset\R^n $ be an open convex set. Then the following hold:
\begin{itemize}
\item[(i)] $ \rec(G)=\rec(\overline{G}) $.
\item[(ii)] $ \rec(G) $ is a closed convex cone and $ 0\in\rec(G) $.
\item[(iii)] $ \rec(G)=\lbrace x\in\R^n:(\exists a\in A)(\forall\lambda\geq 0)\;a+\lambda x\in A\rbrace $.
\item[(iv)] If $ G\subset\R^n $ is unbounded, then $ 1\leq\dim(\spann(\rec(G))) $.
\item[(v)] $ \dim(\spann(\rec(G)))=n $ if and only if $ G $ contains a translation of a cone with non-empty interior.
\end{itemize}
\end{lemma}

We will also need the following special case of \cite[Lemma 5.11]{KZ}:

\begin{lemma}\label{le_kz2}
Let $ m\in\N $, $ m\geq 2 $, and $ \omega\in\M $. Let $ G\subset\R^m $ be an open convex set and $ L:\R^m\to\R^2 $ a linear surjection. Suppose that $ \omega $ is concave, $ L(G)\in\mathcal{G}_2(\omega) $ and $ \rec(L(G))\subset L(\rec(G)) $. Then $ G\in\mathcal{G}_m(\omega) $.
\end{lemma}

The following lemma is an easy consequence of \cite[Theorem 9.1]{R}.

\begin{lemma}\label{le_rock}
Let $ m,n\in\N $, let $ G\subset\R^n $ be an unbounded open convex set and let $ L:\R^n\to\R^m $ be a linear surjection such that
\begin{align*}
\rec(G)\cap L^{-1}(\lbrace 0\rbrace)\subset\rec(G)\cap(-\rec(G)).
\end{align*}
Then $ \rec(L(G))=L(\rec(G)) $.
\end{lemma}

\begin{proof}
By Lemma \ref{le_rec} (i) we have $ \rec(G)=\rec(\overline{G}) $. So $ \rec(\overline{G})\cap L^{-1}(\lbrace 0\rbrace)\subset\rec(\overline{G})\cap(-\rec(\overline{G})) $ and thus \cite[Theorem 9.1]{R} (applied to $ C:=\overline{G} $) implies that $ L(\overline{G}) $ is closed and $ \rec(L(\overline{G}))=L(\rec(\overline{G})) $. Hence we easily obtain that $ L(\overline{G})=\overline{L(G)} $. Therefore, $ \rec(L(\overline{G}))=\rec(L(G)) $ by Lemma \ref{le_rec} (i) and the equality $ \rec(L(G))=L(\rec(G)) $ follows.
\end{proof}

\section{Modulus $ \omega_{\eta} $}

One of the main ingredients of this article is the introduction of modulus $ \omega_{\eta} $:

\begin{definition}\label{de_omegaeta}
Let $ \eta:(0,\infty)\to(0,\infty) $ and $ \omega\in\M $. Then for every $ h>0 $ we set
\begin{align}\label{eq_modulus}
\omega_{\eta}(h):=\inf\bigg\lbrace\sum_{j=1}^{n}\max\bigg\lbrace 1,\frac{x_j-x_{j-1}}{\eta(x_j)}\bigg\rbrace\omega(x_j-x_{j-1})\bigg\rbrace,
\end{align}
where the infimum is taken over all partitions $ 0=x_0<x_1<\cdots<x_n=h $ of $ [0,h] $. Further we set $ \omega_{\eta}(0):=0 $.
\end{definition}

The fact that $ \omega_{\eta} $ is indeed a modulus will be proved in Lemma \ref{le_modulus} (i) below.

\begin{remark}\label{re_f}
The original motivation for $ \omega_{\eta} $ was that the following statement holds:
\begin{itemize}
\item Let $ \eta:(0,\infty)\to(0,\infty) $ be concave, let $ \omega\in\M $ and set
\begin{align*}
G:=\lbrace(x,y)\in\R^2:x>0,\;\vert y\vert<\eta(x)\rbrace.
\end{align*}
Let $ f:G\to\R $ be both $ \omega $-semiconvex and $ \omega $-semiconcave. Then $ f\in C^{1,\omega_{\eta}}(G) $.
\end{itemize}
However, we do not need this statement and therefore not prove it. Note also that $ \omega_{\eta} $ is ``minimal" in the sense of Proposition \ref{pr_main} below which we use substantially at the end of the article.
\end{remark}

\begin{remark}\label{re_trivial}
Let $ \eta:(0,\infty)\to(0,\infty) $ and $ \omega\in\M $. Then
\begin{align*}
\omega_{\eta}(h)\leq\max\bigg\lbrace 1,\frac{h}{\eta(h)}\bigg\rbrace\omega(h),\quad h>0.
\end{align*}
Indeed, taking $ n=1 $ in the sum in \eqref{eq_modulus}, we immediately obtain the estimate.
\end{remark}

\begin{lemma}\label{le_modulus}
Let $ \eta:(0,\infty)\to(0,\infty) $ be concave and let $ \omega\in\M $. Then the following hold:
\begin{itemize}
\item[(i)] $ \omega_{\eta}\in\M $.
\item[(ii)] For every $ x,h>0 $ we have $ \omega_{\eta}(x+h)\leq\omega_{\eta}(x)+\omega_{\eta}(h) $.
\end{itemize}
\end{lemma}

\begin{proof}
(i): Obviously $ \omega_{\eta}(h)\geq 0 $ for every $ h\geq 0 $. By Remark \ref{re_trivial} and Remark \ref{re_concave} (ii), we have
\begin{align*}
\omega_{\eta}(h)\leq\max\bigg\lbrace 1,\frac{h}{\eta(h)}\bigg\rbrace\omega(h)\leq\max\bigg\lbrace 1,\frac{1}{\eta(1)}\bigg\rbrace\omega(h), \quad h\in(0,1),
\end{align*}
and thus $ \lim_{h\to 0+}\omega_{\eta}(h)=\omega_{\eta}(0)=0 $. To prove that $ \omega_{\eta} $ is non-decreasing, it is sufficient to show
\begin{align}\label{eq_monotone}
\omega_{\eta}(x)\leq\omega_{\eta}(x+h),\quad x,h>0.
\end{align}
Let $ x,h>0 $ and let $ 0=x_0<x_1<\cdots<x_n=x+h $ be a partition of $ [0,x+h] $. Then there exists $ m\in\N $ such that $ m\leq n $ and $ x_{m-1}<x\leq x_m $. According to Remark \ref{re_concave} (i), (ii), we have
\begin{align*}
\frac{x-x_{m-1}}{\eta(x)}=\frac{x}{\eta(x)}-\frac{x_{m-1}}{\eta(x)}\leq\frac{x_m}{\eta(x_m)}-\frac{x_{m-1}}{\eta(x_m)}=\frac{x_m-x_{m-1}}{\eta(x_m)}
\end{align*}
and thus
\begin{align*}
\omega_{\eta}(x)&\leq\max\bigg\lbrace 1,\frac{x-x_{m-1}}{\eta(x)}\bigg\rbrace\omega(x-x_{m-1})\\
&+\sum_{j=1}^{m-1}\max\bigg\lbrace 1,\frac{x_j-x_{j-1}}{\eta(x_j)}\bigg\rbrace\omega(x_j-x_{j-1})\\
&\leq\sum_{j=1}^{m}\max\bigg\lbrace 1,\frac{x_j-x_{j-1}}{\eta(x_j)}\bigg\rbrace\omega(x_j-x_{j-1})\\
&\leq\sum_{j=1}^{n}\max\bigg\lbrace 1,\frac{x_j-x_{j-1}}{\eta(x_j)}\bigg\rbrace\omega(x_j-x_{j-1}).
\end{align*}
Hence \eqref{eq_monotone} holds.

(ii): Let $ x,h>0 $ and $ \varepsilon>0 $. Then there exists a partition $ 0=x_0<x_1<\cdots<x_n=x $ of $ [0,x] $ such that
\begin{align*}
\sum_{j=1}^{n}\max\bigg\lbrace 1,\frac{x_j-x_{j-1}}{\eta(x_j)}\bigg\rbrace\omega(x_j-x_{j-1})\leq\omega_{\eta}(x)+\frac{\varepsilon}{2},
\end{align*}
and a partition $ 0=y_0<y_1<\cdots<y_m=h $ of $ [0,h] $ such that
\begin{align*}
\sum_{i=1}^{m}\max\bigg\lbrace 1,\frac{y_i-y_{i-1}}{\eta(y_i)}\bigg\rbrace\omega(y_i-y_{i-1})\leq\omega_{\eta}(h)+\frac{\varepsilon}{2}.
\end{align*}
Hence, for $ x_j:=x+y_{j-n} $ ($ j=n+1,\dots,n+m $) we obtain using Remark \ref{re_concave} (i)
\begin{align*}
&\omega_{\eta}(x+h)\leq\sum_{j=1}^{n+m}\max\bigg\lbrace 1,\frac{x_j-x_{j-1}}{\eta(x_j)}\bigg\rbrace\omega(x_j-x_{j-1})\\
&=\sum_{j=1}^{n}\max\bigg\lbrace 1,\frac{x_j-x_{j-1}}{\eta(x_j)}\bigg\rbrace\omega(x_j-x_{j-1})+\sum_{i=1}^{m}\max\bigg\lbrace 1,\frac{y_i-y_{i-1}}{\eta(x+y_i)}\bigg\rbrace\omega(y_i-y_{i-1})\\
&\leq\sum_{j=1}^{n}\max\bigg\lbrace 1,\frac{x_j-x_{j-1}}{\eta(x_j)}\bigg\rbrace\omega(x_j-x_{j-1})+\sum_{i=1}^{m}\max\bigg\lbrace 1,\frac{y_i-y_{i-1}}{\eta(y_i)}\bigg\rbrace\omega(y_i-y_{i-1})\\
&\leq\omega_{\eta}(x)+\omega_{\eta}(h)+\varepsilon.
\end{align*}
Thus (ii) holds.
\end{proof}

In the rest of this section we will prove several properties of $ \omega_{\eta} $ that we will need further.

\begin{lemma}\label{le_lowerestimate}
Let $ \eta:(0,\infty)\to(0,\infty) $ be non-decreasing and let $ \omega\in\M $ be concave. Then
\begin{align}\label{eq_lowerestimate}
\frac{h\omega(\eta(h))}{\eta(h)}\leq\omega_{\eta}(h),\quad h>0.
\end{align}
\end{lemma}

\begin{proof}
Let $ h>0 $ and $ \varepsilon>0 $. Then there exists a partition $ 0=x_0<x_1<\cdots<x_n=h $ of $ [0,h] $ such that
\begin{align*}
\sum_{j=1}^{n}\max\bigg\lbrace 1,\frac{x_j-x_{j-1}}{\eta(x_j)}\bigg\rbrace\omega(x_j-x_{j-1})\leq\omega_{\eta}(h)+\varepsilon.
\end{align*}
Set
\begin{align*}
J_1:=\lbrace j\in\lbrace 1,\dots,n\rbrace:\eta(x_j)\leq x_j-x_{j-1}\rbrace,\\
J_2:=\lbrace j\in\lbrace 1,\dots,n\rbrace:x_j-x_{j-1}<\eta(x_j)\rbrace.
\end{align*}
Using Remark \ref{re_concave} (ii) we obtain
\begin{align*}
(x_j-x_{j-1})\frac{\omega(\eta(h))}{\eta(h)}&\leq(x_j-x_{j-1})\frac{\omega(\eta(x_j))}{\eta(x_j)}\\
&\leq\frac{x_j-x_{j-1}}{\eta(x_j)}\omega(x_j-x_{j-1})\\
&\leq\max\bigg\lbrace 1,\frac{x_j-x_{j-1}}{\eta(x_j)}\bigg\rbrace\omega(x_j-x_{j-1}),\quad j\in J_1,
\end{align*}
and 
\begin{align*}
(x_j-x_{j-1})\frac{\omega(\eta(h))}{\eta(h)}&\leq(x_j-x_{j-1})\frac{\omega(x_j-x_{j-1})}{x_j-x_{j-1}}\\
&\leq\max\bigg\lbrace 1,\frac{x_j-x_{j-1}}{\eta(x_j)}\bigg\rbrace\omega(x_j-x_{j-1}),\quad j\in J_2.
\end{align*}
Consequently
\begin{align*}
\frac{h\omega(\eta(h))}{\eta(h)}&=\sum_{j\in J_1}(x_j-x_{j-1})\frac{\omega(\eta(h))}{\eta(h)}+\sum_{j\in J_2}(x_j-x_{j-1})\frac{\omega(\eta(h))}{\eta(h)}\\
&\leq\sum_{j\in J_1}\max\bigg\lbrace 1,\frac{x_j-x_{j-1}}{\eta(x_j)}\bigg\rbrace\omega(x_j-x_{j-1})\\
&+\sum_{j\in J_2}\max\bigg\lbrace 1,\frac{x_j-x_{j-1}}{\eta(x_j)}\bigg\rbrace\omega(x_j-x_{j-1})\\
&=\sum_{j=1}^{n}\max\bigg\lbrace 1,\frac{x_j-x_{j-1}}{\eta(x_j)}\bigg\rbrace\omega(x_j-x_{j-1})\leq\omega_{\eta}(h)+\varepsilon.
\end{align*}
Hence \eqref{eq_lowerestimate} holds.
\end{proof}

\begin{lemma}\label{le_est}
Let $ \eta:(0,\infty)\to(0,\infty) $ and $ \omega\in\M $. Then
\begin{align}\label{eq_est}
\omega_{\eta}(x+h)\leq\omega_{\eta}(x)+\max\bigg\lbrace 1,\frac{h}{\eta(x+h)}\bigg\rbrace\omega(h),\quad x,h>0.
\end{align}
\end{lemma}

\begin{proof}
Let $ x,h>0 $ and $ \varepsilon>0 $. Then there exists a partition $ 0=x_0<x_1<\cdots<x_n=x $ of $ [0,x] $ such that
\begin{align*}
\sum_{j=1}^{n}\max\bigg\lbrace 1,\frac{x_j-x_{j-1}}{\eta(x_j)}\bigg\rbrace\omega(x_j-x_{j-1})\leq\omega_{\eta}(x)+\varepsilon.
\end{align*}
Hence, for $ x_{n+1}:=x+h $ we have
\begin{align*}
\omega_{\eta}(x+h)&\leq\sum_{j=1}^{n+1}\max\bigg\lbrace 1,\frac{x_j-x_{j-1}}{\eta(x_j)}\bigg\rbrace\omega(x_j-x_{j-1})\\
&=\max\bigg\lbrace 1,\frac{h}{\eta(x+h)}\bigg\rbrace\omega(h)+\sum_{j=1}^{n}\max\bigg\lbrace 1,\frac{x_j-x_{j-1}}{\eta(x_j)}\bigg\rbrace\omega(x_j-x_{j-1})\\
&\leq\omega_{\eta}(x)+\max\bigg\lbrace 1,\frac{h}{\eta(x+h)}\bigg\rbrace\omega(h)+\varepsilon.
\end{align*}
Consequently \eqref{eq_est} holds.
\end{proof}

\begin{lemma}\label{le_est2}
Let $ \eta:(0,\infty)\to(0,\infty) $ be non-decreasing and let $ \omega\in\M $. Then 
\begin{align}\label{eq_est2}
\omega_{\eta}(x+h)\leq\omega_{\eta}(x)+2\frac{h\omega(\eta(x))}{\eta(x)},\quad x,h>0,h\geq\frac{\eta(x)}{2}.
\end{align}
\end{lemma}

\begin{proof}
Let $ x,h>0 $, $ h\geq\eta(x)/2 $, and $ \varepsilon>0 $. Then there exists a partition $ 0=x_0<x_1<\cdots<x_n=x $ of $ [0,x] $ such that
\begin{align*}
\sum_{j=1}^{n}\max\bigg\lbrace 1,\frac{x_j-x_{j-1}}{\eta(x_j)}\bigg\rbrace\omega(x_j-x_{j-1})\leq\omega_{\eta}(x)+\varepsilon.
\end{align*}
Denote by $ m $ the upper integer part of $ h/\eta(x) $ and set $ x_j:=x+(j-n)h/m $ for each $ j=n+1,\dots,n+m $. Then
\begin{align*}
\frac{m}{2}\leq\frac{h}{\eta(x)}\leq m
\end{align*}
and $ h/m\leq\eta(x)\leq\eta(x_j) $ for each $ j=n+1,\dots,n+m $. Thus
\begin{align*}
\omega_{\eta}(x+h)&\leq\sum_{j=1}^{n+m}\max\bigg\lbrace 1,\frac{x_j-x_{j-1}}{\eta(x_j)}\bigg\rbrace\omega(x_j-x_{j-1})\\
&=\sum_{j=1}^{n}\max\bigg\lbrace 1,\frac{x_j-x_{j-1}}{\eta(x_j)}\bigg\rbrace\omega(x_j-x_{j-1})\\
&+\sum_{j=n+1}^{n+m}\max\bigg\lbrace 1,\frac{\frac{h}{m}}{\eta(x_j)}\bigg\rbrace\omega\bigg(\frac{h}{m}\bigg)\\
&\leq\omega_{\eta}(x)+\varepsilon+\sum_{j=1}^{m}\omega\bigg(\frac{h}{m}\bigg)\leq\omega_{\eta}(x)+\varepsilon+\sum_{j=1}^{m}\omega(\eta(x))\\
&=\omega_{\eta}(x)+m\omega(\eta(x))+\varepsilon\leq\omega_{\eta}(x)+2\frac{h\omega(\eta(x))}{\eta(x)}+\varepsilon.
\end{align*}
Hence \eqref{eq_est2} holds.
\end{proof}

\section{The case $ \R^2 $}

We begin with three lemmas that we need to prove important Proposition \ref{pr_main} below.

\begin{lemma}\label{le_g1}
Let $ \eta:(0,\infty)\to (0,\infty) $ and $ \omega\in\M $ be both concave, let $ g:(0,\infty)\to\R $ be differentiable and let $ C>0 $. Suppose that
\begin{align}
\vert g(x+h)-g(x)\vert&\leq C\max\bigg\lbrace 1,\frac{h}{\eta(x+h)}\bigg\rbrace\omega(h),\quad x,h>0,\label{eq_g1_1}\\
\eta(x)\vert g'(x)\vert&\leq C\omega(\eta(x)),\quad x>0,\label{eq_g1_2}\\
\eta(x)\vert g'(x)-g'(x+h)\vert&\leq C\omega(h),\quad x,h>0,\label{eq_g1_3}
\end{align}
and set
\begin{align*}
G&:=\lbrace(x,y)\in\R^2:x>0,\;\vert y\vert<\eta(x)\rbrace,\\
f(x,y)&:=g(x)y,\quad (x,y)\in G.
\end{align*}
Then the function $ f $ is both $ (5C\omega) $-semiconvex and $ (5C\omega) $-semiconcave.
\end{lemma}

\begin{proof}
By \cite[Lemma 2.3]{KZ} it is sufficient to show that
\begin{multline*}
\vert f'(x+h,y+k)[h,k]-f'(x,y)[h,k]\vert\\\leq 5C\vert (h,k)\vert\omega(\vert (h,k)\vert),\quad (x,y),(x+h,y+k)\in G, h\geq 0.
\end{multline*}
Let $ (x,y),(x+h,y+k)\in G $, $ h\geq 0 $. Then, according to Remark \ref{re_concave} (i), we have
\begin{align}\label{eq_kest}
\vert k\vert\leq\vert y+k\vert+\vert y\vert\leq\eta(x+h)+\eta(x)\leq 2\eta(x+h)
\end{align}
and thus by \eqref{eq_g1_2} and Remark \ref{re_concave} (ii)
\begin{align*}
\vert k\vert\vert g'(x+h)\vert&=\vert k\vert\frac{\eta(x+h)\vert g'(x+h)\vert}{\eta(x+h)}\leq C\vert k\vert\frac{\omega(\eta(x+h))}{\eta(x+h)}\\
&\leq C\vert k\vert\frac{\omega\big(\frac{\vert k\vert}{2}\big)}{\frac{\vert k\vert}{2}}\leq 2C\omega(\vert k\vert).
\end{align*}
Further, by \eqref{eq_g1_3} we have
\begin{align*}
\vert y\vert\vert g'(x+h)-g'(x)\vert\leq\eta(x)\vert g'(x+h)-g'(x)\vert\leq C\omega(h).
\end{align*}
Consequently
\begin{multline}\label{eq_h}
\vert g'(x+h)(y+k)-g'(x)y\vert\leq\vert y\vert\vert g'(x+h)-g'(x)\vert+\vert k\vert\vert g'(x+h)\vert\\
\leq C\omega(h)+2C\omega(\vert k\vert)\leq 3C\omega(\vert (h,k)\vert).
\end{multline}
Since by \eqref{eq_kest} we have
\begin{align*}
\vert k\vert\max\bigg\lbrace 1,\frac{h}{\eta(x+h)}\bigg\rbrace=\max\bigg\lbrace\vert k\vert,\frac{\vert k\vert}{\eta(x+h)}h\bigg\rbrace\leq 2\max\lbrace\vert k\vert,h\rbrace,
\end{align*}
we obtain by \eqref{eq_g1_1}
\begin{multline}\label{eq_k}
\vert k\vert\vert g(x+h)-g(x)\vert\leq C\vert k\vert\max\bigg\lbrace 1,\frac{h}{\eta(x+h)}\bigg\rbrace\omega(h)\\
\leq 2C\max\lbrace\vert k\vert,h\rbrace\omega(h)\leq 2C\vert (h,k)\vert\omega(\vert (h,k)\vert).
\end{multline}
Hence \eqref{eq_h} and \eqref{eq_k} give
\begin{align*}
&\vert f'(x+h,y+k)[h,k]-f'(x,y)[h,k]\vert\\
&\leq\vert h\vert\bigg\vert\frac{\partial f}{\partial x}(x+h,y+k)-\frac{\partial f}{\partial x}(x,y)\bigg\vert+\vert k\vert\bigg\vert\frac{\partial f}{\partial y}(x+h,y+k)-\frac{\partial f}{\partial y}(x,y)\bigg\vert\\
&=\vert h\vert\vert g'(x+h)(y+k)-g'(x)y\vert+\vert k\vert\vert g(x+h)-g(x)\vert\\
&\leq 3C\vert (h,k)\vert\omega(\vert (h,k)\vert)+2C\vert (h,k)\vert\omega(\vert (h,k)\vert)=5C\vert (h,k)\vert\omega(\vert (h,k)\vert).
\end{align*}
\end{proof}

\begin{lemma}\label{le_envelope}
Let $ \eta:(0,\infty)\to(0,\infty) $ and $ \omega\in\M $ be both concave. Define the function $ \psi $ by
\begin{multline*}
\psi(x):=\sup\lbrace\lambda\omega_{\eta}(x_1)+(1-\lambda)\omega_{\eta}(x_2):x_1,x_2\geq 0,\lambda\in[0,1],\\
x=\lambda x_1+(1-\lambda)x_2\rbrace,\quad x\geq 0.
\end{multline*}
Then $ \psi\in\M $, $ \psi $ is concave, $ \omega_{\eta}\leq\psi\leq 2\omega_{\eta} $ and
\begin{align}\label{eq_envelope}
\psi(t+\eta(t))\leq\psi(t)+2\omega(\eta(t)),\quad t>0,\eta(t)\leq t.
\end{align}
\end{lemma}

\begin{proof}
By Remark \ref{re_additive} (ii) and Lemma \ref{le_modulus}, there exists a concave $ \varphi\in\M $ such that $ 0\leq\omega_{\eta}\leq\varphi\leq 2\omega_{\eta} $. Note that $ \psi $ is the upper concave envelope of $ \omega_{\eta} $ (see \cite[Corollary 17.1.5]{R}). In particular, $ \psi $ is concave and $ 0\leq\omega_{\eta}\leq\psi\leq\varphi $. Thus $ \psi\leq 2\omega_{\eta} $, $ \lim_{x\to 0+}\psi(x)=0 $ and it follows from Remark \ref{re_concave} (i) that $ \psi $ is non-decreasing. Therefore, $ \psi\in\M $.

Let $ t>0 $, $ x_1,x_2\geq 0 $ and $ \lambda\in[0,1] $ be such that $ \eta(t)\leq t $, $ x_1\leq x_2 $ and $ \lambda x_1+(1-\lambda)x_2=t+\eta(t) $. Then to prove \eqref{eq_envelope} it is sufficient to show
\begin{align*}
\lambda\omega_{\eta}(x_1)+(1-\lambda)\omega_{\eta}(x_2)\leq\psi(t)+2\omega(\eta(t)).
\end{align*}
Note also that $ \eta $ is non-decreasing by Remark \ref{re_concave} (i).

If $ t\leq x_1 $, then $ \eta(t)\leq t\leq x_1\leq x_2 $ and thus by Lemma \ref{le_est} (with $ h=\eta(t) $ and $ x=x_1-\eta(t) $ or  $ x=x_2-\eta(t) $)
\begin{align*}
\omega_{\eta}(x_i)&\leq\omega_{\eta}(x_i-\eta(t))+\max\bigg\lbrace 1,\frac{\eta(t)}{\eta(x_i)}\bigg\rbrace\omega(\eta(t))\\
&=\omega_{\eta}(x_i-\eta(t))+\omega(\eta(t)),\quad i=1,2.
\end{align*}
Hence
\begin{align*}
\lambda\omega_{\eta}(x_1)+(1-\lambda)\omega_{\eta}(x_2)&\leq\lambda\omega_{\eta}(x_1-\eta(t))+(1-\lambda)\omega_{\eta}(x_2-\eta(t))+\omega(\eta(t))\\
&\leq\lambda\psi(x_1-\eta(t))+(1-\lambda)\psi(x_2-\eta(t))+\omega(\eta(t))\\
&\leq\psi(\lambda(x_1-\eta(t))+(1-\lambda)(x_2-\eta(t)))+\omega(\eta(t))\\
&=\psi(t)+\omega(\eta(t))\leq\psi(t)+2\omega(\eta(t)).
\end{align*}

Now suppose that $ x_1\leq t $. Then $ \lambda<1 $ and therefore we may set
\begin{align*}
h:=\frac{\eta(t)}{1-\lambda},\;x:=x_2-\frac{\eta(t)}{1-\lambda}.
\end{align*}
Then $ (1-\lambda)t=t-\lambda t\leq t-\lambda x_1=(1-\lambda)x_2-\eta(t)=(1-\lambda)x $ and thus
\begin{align}\label{eq_space}
t\leq x.
\end{align}
In particular, $ 0<x $. If $ h\leq\eta(x+h) $, then by Lemma \ref{le_est} and Remark \ref{re_concave} (iii)
\begin{align*}
\omega_{\eta}(x_2)&=\omega_{\eta}(x+h)\leq\omega_{\eta}(x)+\max\bigg\lbrace 1,\frac{h}{\eta(x+h)}\bigg\rbrace\omega(h)\\
&=\omega_{\eta}(x)+\omega\bigg(\frac{\eta(t)}{1-\lambda}\bigg)\leq\omega_{\eta}(x)+\frac{\omega(\eta(t))}{1-\lambda}
\end{align*}
and thus
\begin{align*}
\lambda\omega_{\eta}(x_1)+(1-\lambda)\omega_{\eta}(x_2)&\leq\lambda\omega_{\eta}(x_1)+(1-\lambda)\omega_{\eta}(x)+\omega(\eta(t))\\
&\leq\lambda\psi(x_1)+(1-\lambda)\psi(x)+\omega(\eta(t))\\
&\leq\psi(\lambda x_1+(1-\lambda)x)+\omega(\eta(t))\\
&=\psi(t)+\omega(\eta(t))\leq\psi(t)+2\omega(\eta(t)).
\end{align*}
If $ \eta(x+h)\leq h $, then by Lemma \ref{le_est2}, Remark \ref{re_concave} (ii) and \eqref{eq_space} we have
\begin{align*}
\omega_{\eta}(x_2)&=\omega_{\eta}(x+h)\leq\omega_{\eta}(x)+2h\frac{\omega(\eta(x))}{\eta(x)}\\
&\leq\omega_{\eta}(x)+2h\frac{\omega(\eta(t))}{\eta(t)}=\omega_{\eta}(x)+2\frac{\omega(\eta(t))}{1-\lambda}
\end{align*}
and therefore
\begin{align*}
\lambda\omega_{\eta}(x_1)+(1-\lambda)\omega_{\eta}(x_2)&\leq\lambda\omega_{\eta}(x_1)+(1-\lambda)\omega_{\eta}(x)+2\omega(\eta(t))\\
&\leq\lambda\psi(x_1)+(1-\lambda)\psi(x)+2\omega(\eta(t))\\
&\leq\psi(\lambda x_1+(1-\lambda)x)+2\omega(\eta(t))\\
&=\psi(t)+2\omega(\eta(t)).
\end{align*}
\end{proof}

\begin{lemma}\label{le_g2}
Let $ \eta:(0,\infty)\to (0,\infty) $ and $ \psi\in\M $ be both concave, let $ \omega\in\M $ and let $ C>0 $. Suppose that
\begin{align}
\psi(t+\eta(t))-\psi(t)&\leq C\omega(\eta(t)),\quad t>0,\;\eta(t)\leq t,\label{eq_g2_1}\\
\psi(h)&\leq C\max\bigg\lbrace 1,\frac{h}{\eta(h)}\bigg\rbrace\omega(h),\quad h>0.\label{eq_g2_2}
\end{align}
Then
\begin{align*}
\psi(x+h)-\psi(x)&\leq 2C\max\bigg\lbrace 1,\frac{h}{\eta(x+h)}\bigg\rbrace\omega(h),\quad x,h>0.
\end{align*}
\end{lemma}

\begin{proof}
Let $ x,h>0 $. We will distinguish four cases.

(a) $ x\leq h $: Then $ \eta(x+h)\leq\eta(2h)\leq 2\eta(h) $ by Remark \ref{re_concave} (i), (iii). And so, using  Remark \ref{re_additive} (i) and \eqref{eq_g2_2}, we obtain
\begin{align*}
&\psi(x+h)-\psi(x)\leq\psi(h)\leq C\max\bigg\lbrace 1,\frac{h}{\eta(h)}\bigg\rbrace\omega(h)\\
&\leq C\max\bigg\lbrace 1,\frac{2h}{\eta(x+h)}\bigg\rbrace\omega(h)\leq 2C\max\bigg\lbrace 1,\frac{h}{\eta(x+h)}\bigg\rbrace\omega(h).
\end{align*}

(b) $ h\leq\eta(h) $: Then, according to Remark \ref{re_additive} (i) and \eqref{eq_g2_2}, we have
\begin{align*}
\psi(x+h)-\psi(x)&\leq\psi(h)\leq C\max\bigg\lbrace 1,\frac{h}{\eta(h)}\bigg\rbrace\omega(h)=C\omega(h)\\
&\leq 2C\max\bigg\lbrace 1,\frac{h}{\eta(x+h)}\bigg\rbrace\omega(h).
\end{align*}

(c) $ \eta(x)\leq h\leq x $: Then $ \eta(x+h)\leq\eta(2x)\leq 2\eta(x) $ by Remark \ref{re_concave} (i), (iii). Thus, using the concavity of $ \psi $ and \eqref{eq_g2_1}, we obtain
\begin{align*}
&\psi(x+h)-\psi(x)=h\frac{\psi(x+h)-\psi(x)}{h}\leq h\frac{\psi(x+\eta(x))-\psi(x)}{\eta(x)}\\
&\leq C\frac{h}{\eta(x)}\omega(\eta(x))\leq 2C\frac{h}{\eta(x+h)}\omega(h)\leq 2C\max\bigg\lbrace 1,\frac{h}{\eta(x+h)}\bigg\rbrace\omega(h).
\end{align*}

(d) Suppose that none of cases (a), (b) or (c) holds: Then clearly $ h\leq x $ and $ \eta(h)\leq h\leq\eta(x) $. Since $ \eta $ is continuous, there exists $ t\in [h,x] $ such that $ \eta(t)=h $. Hence, using the concavity of $ \psi $ and \eqref{eq_g2_1}, we obtain
\begin{align*}
\psi(x+h)-\psi(x)&=\eta(t)\frac{\psi(x+\eta(t))-\psi(x)}{\eta(t)}\leq\eta(t)\frac{\psi(t+\eta(t))-\psi(t)}{\eta(t)}\\&\leq C\omega(\eta(t))\leq 2C\max\bigg\lbrace 1,\frac{h}{\eta(x+h)}\bigg\rbrace\omega(h).
\end{align*}
\end{proof}

Now we can prove the crucial proposition of this article, which we will need to prove Theorem \ref{th_main} below. However, we believe that it is also interesting in itself.

\begin{proposition}\label{pr_main}
Let $ \eta:(0,\infty)\to(0,\infty) $ and $ \omega\in\M $ be both concave. Suppose that $ \omega(x)>0 $ for all $ x>0 $, and set
\begin{align*}
G:=\lbrace(x,y)\in\R^2:x>0,\;\vert y\vert<\eta(x)\rbrace .
\end{align*}
Then there exists a concave and strictly positive $ g\in C^1((0,\infty)) $ such that the function $ f $ defined by
\begin{align*}
f(x,y)=g(x)y,\quad (x,y)\in G,
\end{align*}
has the following properties:
\begin{itemize}
\item[(i)] $ f $ is both $ \omega $-semiconvex and $ \omega $-semiconcave.
\item[(ii)] If $ \varphi\in\M $ satisfies $ \liminf_{x\to\infty}\varphi(x)/\omega_{\eta}(x)=0 $, then there is no $ C>0 $ such that
\begin{align}\label{eq_varphi1}
\bigg\vert\frac{\partial f}{\partial y}(x,0)-\frac{\partial f}{\partial y}(1,0)\bigg\vert\leq C\varphi(\vert x-1\vert),\quad x>0,
\end{align}
in particular, $ f\notin C^{1,\varphi}(G) $.
\end{itemize}
\end{proposition}

\begin{proof}
Firstly, note that $ \eta $ is non-decreasing by Remark \ref{re_concave} (i). So it follows from Lemma \ref{le_lowerestimate} that $ \omega_{\eta}(x)>0 $ for all $ x>0 $. Hence by Lemma \ref{le_envelope} there exists a concave $ \psi\in\M $ such that \eqref{eq_g2_1} holds with $ C=2 $ and 
\begin{align}\label{eq_equivalent}
0<\omega_{\eta}(x)\leq\psi(x)\leq 2\omega_{\eta}(x),\quad x>0.
\end{align}
Further, it follows from \eqref{eq_equivalent} and Remark \ref{re_trivial} that \eqref{eq_g2_2} holds with $ C=2 $. Hence by Lemma \ref{le_g2} we have
\begin{align}\label{eq_uniform}
\psi(x+h)-\psi(x)&\leq 4\max\bigg\lbrace 1,\frac{h}{\eta(x+h)}\bigg\rbrace\omega(h),\quad x,h>0.
\end{align}
Note that since $ \psi $ is concave, we have
\begin{align*}
\psi(x)=\int_{0}^{x}\psi'_+(s)ds,\quad x\geq 0.
\end{align*}
Since $ \psi\in\M $ and $ \psi(x)>0 $ for every $ x>0 $, there exists $ a>0 $ such that
\begin{align}\label{eq_constant}
\psi'_+(4a)>0.
\end{align}
By Remark \ref{re_concave} (ii) we have
\begin{align}\label{eq_defc}
\frac{\eta(x)}{x}\leq\frac{\eta(a)}{a}\leq q:=\max\bigg\lbrace 1,\frac{\eta(a)}{a}\bigg\rbrace,\quad x\geq a.
\end{align}
Denote by $ \delta $ the real function on $ [a,\infty) $ that is affine on each $ [2^{n-1}a,2^na] $ ($ n\in\N $) and 
\begin{align*}
\delta(2^na)=\psi'_+(2^na),\quad n\in\N\cup\lbrace 0\rbrace.
\end{align*}
Then $ \delta $ is continuous and non-negative. Since $ \psi'_+ $ is non-increasing, $ \delta $ is also non-increasing. Hence for every $ n\in\N $ and $ x\in [2^{n-1}a,2^na] $ we have
\begin{align*}
\psi'_+(2x)\leq\psi'_+(2^na)=\delta(2^na)\leq\delta(x)\leq\delta(2^{n-1}a)=\psi'_+(2^{n-1}a)\leq\psi'_+\bigg(\frac{x}{2}\bigg)
\end{align*}
and thus
\begin{align}\label{eq_psideltapsi}
\psi'_+(2x)\leq\delta(x)\leq\psi'_+\bigg(\frac{x}{2}\bigg),\quad x\geq a.
\end{align}

We set $ b:=1280q^2 $ and
\begin{align*}
g(x):=\frac{1}{b}\int_{a}^{a+x}\delta(s)ds,\quad x>0.
\end{align*}
Then $ g $ is concave and $ g\in C^1((0,\infty)) $. Furthermore, for all $ x>0 $ we obtain using \eqref{eq_psideltapsi} and \eqref{eq_constant} 
\begin{align*}
g(x)&\geq\frac{1}{b}\int_{a}^{\min\lbrace a+x,2a\rbrace}\delta(s)ds\geq\frac{1}{b}\int_{a}^{\min\lbrace a+x,2a\rbrace}\psi'_+(2s)ds\\
&\geq\frac{1}{b}\int_{a}^{\min\lbrace a+x,2a\rbrace}\psi'_+(4a)ds>0.
\end{align*}

(i): Note that $ g' $ is non-increasing and, by Remark \ref{re_concave} (i), also non-negative. Note also that $ 8/b\leq 8q/b\leq 1/5 $. Hence by Lemma \ref{le_g1} (with $ C=1/5 $) it is sufficient to show that
\begin{align}
g(x+h)-g(x)&\leq\frac{8}{b}\max\bigg\lbrace 1,\frac{h}{\eta(x+h)}\bigg\rbrace\omega(h),\quad x,h>0,\label{eq_th_1}\\
\eta(x)g'(x)&\leq\frac{8q}{b}\omega(\eta(x)),\quad x>0,\label{eq_th_2}\\
\eta(x)(g'(x)-g'(x+h))&\leq\frac{1}{5}\omega(h),\quad x,h>0.\label{eq_th_3}
\end{align}

Inequality \eqref{eq_th_1}: For all $ x,h>0 $ we have by \eqref{eq_psideltapsi}, \eqref{eq_uniform} and Remark \ref{re_concave} (iii)
\begin{align*}
&g(x+h)-g(x)\\
&=\frac{1}{b}\int_{a+x}^{a+x+h}\delta(s)ds\leq\frac{1}{b}\int_{a+x}^{a+x+h}\psi'_+\bigg(\frac{s}{2}\bigg)ds=\frac{2}{b}\int_{\frac{a+x}{2}}^{\frac{a+x+h}{2}}\psi'_+(t)dt\\
&=\frac{2}{b}\bigg(\psi\bigg(\frac{a+x+h}{2}\bigg)-\psi\bigg(\frac{a+x}{2}\bigg)\bigg)\leq\frac{8}{b}\max\bigg\lbrace 1,\frac{\frac{h}{2}}{\eta(\frac{a+x+h}{2})}\bigg\rbrace\omega\bigg(\frac{h}{2}\bigg)\\
&\leq\frac{8}{b}\max\bigg\lbrace 1,\frac{\frac{h}{2}}{\frac{\eta(a+x+h)}{2}}\bigg\rbrace\omega\bigg(\frac{h}{2}\bigg)\leq\frac{8}{b}\max\bigg\lbrace 1,\frac{h}{\eta(x+h)}\bigg\rbrace\omega(h).
\end{align*}

Inequality \eqref{eq_th_2}: Let $ x>0 $. Then we have by \eqref{eq_defc}
\begin{align*}
t:=\frac{\eta(x)}{2q}\leq\frac{\eta(a+x)}{2q}\leq\frac{a+x}{2}
\end{align*}
and by Remark \ref{re_concave} (iii)
\begin{align*}
\frac{t}{\eta(\frac{a+x}{2})}=\frac{\eta(x)}{2q\eta(\frac{a+x}{2})}\leq\frac{\eta(\frac{x}{2})}{q\eta(\frac{a+x}{2})}\leq\frac{1}{q}\leq 1.
\end{align*}
Consequently, using the concavity of $ \psi $, \eqref{eq_psideltapsi} and \eqref{eq_uniform}, we obtain
\begin{align*}
\eta(x)g'(x)&=\frac{1}{b}\eta(x)\delta(a+x)\leq\frac{1}{b}\eta(x)\psi'_+\bigg(\frac{a+x}{2}\bigg)\\
&\leq\frac{\eta(x)}{b}\dfrac{\psi(\frac{a+x}{2})-\psi(\frac{a+x}{2}-t)}{t}=\frac{2q}{b}\bigg(\psi\bigg(\frac{a+x}{2}\bigg)-\psi\bigg(\frac{a+x}{2}-t\bigg)\bigg)\\
&\leq\frac{8q}{b}\max\bigg\lbrace 1,\frac{t}{\eta(\frac{a+x}{2})}\bigg\rbrace\omega(t)=\frac{8q}{b}\omega\bigg(\frac{\eta(x)}{2q}\bigg)\leq\frac{8q}{b}\omega(\eta(x)).
\end{align*}

Inequality \eqref{eq_th_3}: By \eqref{eq_th_2} and Remark \ref{re_concave} (iii) we have
\begin{multline}\label{eq_ng}
\eta(u)(g'(u)-g'(v))\leq\eta(u)g'(u)\leq\frac{8q}{b}\omega(\eta(u))\leq\frac{8q}{b}\omega(4q(v-u))\\\leq\frac{32q^2}{b}\omega(v-u)=\frac{1}{40}\omega(v-u),\quad 0<u<v,\eta(u)\leq 4q(v-u).
\end{multline}
Set
\begin{align*}
x_n:=2^na-a,\quad n\in\N\cup\lbrace 0\rbrace.
\end{align*}
Then by Remark \ref{re_concave} (iii) we have
\begin{align*}
\eta(x_n)\leq 2\eta\bigg(\frac{x_n}{2}\bigg)\leq 2\eta\bigg(\frac{x_{n-1}+x_n}{2}\bigg),\quad n\in\N,
\end{align*}
and by \eqref{eq_defc}
\begin{align*}
\eta\bigg(\frac{x_{n-1}+x_n}{2}\bigg)\leq\eta(x_n)\leq qx_n\leq q\cdot 2^na=4q\frac{x_n-x_{n-1}}{2},\quad n\in\N.
\end{align*}
Hence by \eqref{eq_ng} (for every $ n\in\N $ applied with $ u=(x_{n-1}+x_n)/2 $ and $ v=x_n $) we obtain
\begin{multline}\label{eq_affine}
q_n:=\frac{g'(\frac{x_{n-1}+x_n}{2})-g'(x_n)}{x_n-\frac{x_{n-1}+x_n}{2}}\leq
\frac{1}{40\eta(\frac{x_{n-1}+x_n}{2})}\frac{\omega(\frac{x_n-x_{n-1}}{2})}{\frac{x_n-x_{n-1}}{2}}\\
\leq\frac{1}{20\eta(\frac{x_{n-1}+x_n}{2})}\frac{\omega(x_n-x_{n-1})}{x_n-x_{n-1}}\leq\frac{1}{10\eta(x_n)}\frac{\omega(x_n-x_{n-1})}{x_n-x_{n-1}},\quad n\in\N.
\end{multline}
Since $ g' $ is affine on $ (x_0,x_1] $ and on each $ [x_{n-1},x_n] $ ($ n\in\N $, $ n>1 $), we obtain by \eqref{eq_affine} and Remark \ref{re_concave} (ii)
\begin{multline}\label{eq_eta}
\eta(u)(g'(u)-g'(v))=\eta(u)q_n(v-u)\leq\frac{\eta(u)}{10\eta(x_n)}\frac{\omega(x_n-x_{n-1})}{x_n-x_{n-1}}(v-u)\\
\leq\frac{\eta(u)}{10\eta(x_n)}\frac{\omega(v-u)}{v-u}(v-u)\leq\frac{1}{10}\omega(v-u),\quad n\in\N,u,v\in[x_{n-1},x_n],0<u<v.
\end{multline}
Let $ x,h>0 $. Then there exists $ n\in\N $ such that $ x_{n-1}\leq x\leq x_n $. If $ \eta(x)\leq4qh $ or $ x+h\leq x_n $, then \eqref{eq_th_3} holds by \eqref{eq_ng} or \eqref{eq_eta}. Hence we may further suppose that $ 4qh\leq\eta(x) $ and $ x_n\leq x+h $. Then we have by \eqref{eq_defc}
\begin{align*}
x+h\leq x+\frac{\eta(x)}{4q}\leq x_n+\frac{\eta(x_n)}{4q}\leq x_n+\frac{x_n}{4}\leq 2x_n\leq x_{n+1}
\end{align*}
and thus, using two-times \eqref{eq_eta}, we obtain
\begin{align*}
\eta(x)(g'(x)-g'(x+h))&=\eta(x)(g'(x)-g'(x_n))+\eta(x)(g'(x_n)-g'(x+h))\\
&\leq\eta(x)(g'(x)-g'(x_n))+\eta(x_n)(g'(x_n)-g'(x+h))\\
&\leq\frac{1}{10}\omega(x_n-x)+\frac{1}{10}\omega(x+h-x_n)\leq\frac{1}{5}\omega(h).
\end{align*}

(ii): Let $ \varphi\in\M $ satisfy $ \liminf_{x\to\infty}\varphi(x)/\omega_{\eta}(x)=0 $. Suppose to the contrary that there exists $ C>0 $ such that \eqref{eq_varphi1} holds. Hence
\begin{align}\label{eq_varphi2}
\vert g(x)-g(1)\vert=\bigg\vert\frac{\partial f}{\partial y}(x,0)-\frac{\partial f}{\partial y}(1,0)\bigg\vert\leq C\varphi(\vert x-1\vert),\quad x>0.
\end{align}
If $ \varphi(x)=0 $ for all $ x>0 $, then \eqref{eq_varphi2} implies that $ g $ is constant, and thus by \eqref{eq_psideltapsi}
\begin{align*}
\psi'_+(4a)\leq\delta(2a)=bg'(a)=0
\end{align*}
which is a contradiction with \eqref{eq_constant}. So we may further suppose that $ \varphi(x_0)>0 $ for some $ x_0>0 $. Then
\begin{align*}
0<\frac{\varphi(x_0)}{\omega_{\eta}(x)}\leq\frac{\varphi(x)}{\omega_{\eta}(x)},\quad x>x_0,
\end{align*}
and thus $ \liminf_{x\to\infty}\varphi(x_0)/\omega_{\eta}(x)=0 $. Since $ \omega_{\eta} $ is non-decreasing by Lemma \ref{le_modulus} (i), it follows that $ \lim_{x\to\infty}\omega_{\eta}(x)=\infty $. Now by \eqref{eq_psideltapsi} and \eqref{eq_equivalent} we have
\begin{align*}
2b(g(x)-g(1))&=2\int_{a+1}^{a+x}\delta(s)ds\geq 2\int_{a+1}^{a+x}\psi'_+(2s)ds=\int_{2(a+1)}^{2(a+x)}\psi'_+(t)dt\\
&=\psi(2(a+x))-\psi(2(a+1))\geq\psi(x)-\psi(2(a+1))\\
&\geq\omega_{\eta}(x)-\psi(2(a+1)),\quad x>1,
\end{align*}
and thus by \eqref{eq_varphi2}
\begin{align*}
\omega_{\eta}(x)-\psi(2(a+1))\leq 2b(g(x)-g(1))\leq 2bC\varphi(x-1)\leq 2bC\varphi(x),\quad x>1.
\end{align*}
Hence
\begin{align*}
1-\frac{\psi(2(a+1))}{\omega_{\eta}(x)}\leq 2bC\frac{\varphi(x)}{\omega_{\eta}(x)},\quad x>1,
\end{align*}
which is a contradiction with $ \liminf_{x\to\infty}\varphi(x)/\omega_{\eta}(x)=0 $ and $ \lim_{x\to\infty}\omega_{\eta}(x)=\infty $.
\end{proof}

\section{The case $ \R^n $}

The main result (Theorem \ref{th_main}) of the article works with $ \omega\in\M $ that satisfy 
\begin{align}\label{de_condition}
\inf_{n\in\N}\bigg(\liminf_{h\to\infty}\frac{\omega(h)}{n\omega\big(\frac{h}{n}\big)}\bigg)=0.\tag{*} 
\end{align}
An immediate consequence of Theorem \ref{th_main} is Theorem \ref{th_1}, since we have:
\begin{equation}\label{eq_condition}
\textnormal{If }\alpha\in (0,1)\textnormal{ and }\omega(h)=h^{\alpha}\textnormal{ for all }h\geq 0,\textnormal{ then }\eqref{de_condition}\textnormal{ holds}.
\end{equation}
Indeed, if $ \alpha\in (0,1) $ and $ \omega(h)=h^{\alpha} $ for all $ h\geq 0 $, then
\begin{align*}
\liminf_{h\to\infty}\frac{\omega(h)}{n\omega\big(\frac{h}{n}\big)}=\frac{1}{n^{1-\alpha}} ,\quad n\in\N,
\end{align*}
and thus \eqref{de_condition} holds.

The following remark concerns the validity of \eqref{de_condition} for some other $ \omega\in\M $.

\begin{remark}\label{re_condition}
\hfill
\begin{itemize}
\item[(i)] If $ \alpha\in [0,1) $ and $ \beta>0 $, then there exist $ C\geq 0 $, $ h_0>1 $ and a concave $ \omega\in\M $ such that $ \omega(h)=h^{\alpha}\log^{\beta}(h)+C $ for all $ h\geq h_0 $ and \eqref{de_condition} holds.
\item[(ii)] If $ \omega(h)=h $ for all $ h\geq 0 $, then \eqref{de_condition} doesn't hold.
\item[(iii)] If $ \beta>0 $, then there exist $ C\geq 0 $, $ h_0>1 $ and a concave $ \omega\in\M $ such that $ \omega(h)=h/\log^{\beta}(h)+C $ for all $ h\geq h_0 $ and   \eqref{de_condition} doesn't hold.
\end{itemize}
All these facts can be easily proved.
\end{remark}

The main property of concave $ \omega\in\M $ satisfying \eqref{de_condition} is that assertion (iii) of the following lemma holds.

\begin{lemma}\label{le_condition}
Let $ \omega\in\M $ be concave and satisfy \eqref{de_condition}. Then the following hold: 
\begin{itemize}
\item[(i)] $ \omega(h)>0 $ for all $ h>0 $ and $ \liminf_{h\to 0+}\omega(h)/h>0 $.
\item[(ii)] $ \lim_{h\to\infty}\omega(h)/h=0 $.
\item[(iii)] If $ \eta:(0,\infty)\to(0,\infty) $ is non-decreasing and $ \lim_{h\to\infty}\eta(h)/h=0 $, then 
\begin{align}\label{eq_liminf}
\liminf_{h\to\infty}\frac{\omega(h)}{\omega_{\eta}(h)}=0.
\end{align}
\end{itemize}
\end{lemma}

\begin{proof}
(i) It follows easily from \eqref{de_condition} and Remark \ref{re_concave} (ii).

(ii): Firstly, the limit exists by Remark \ref{re_concave} (ii). Further, using (i) and Remark \ref{re_concave} (iii), we obtain
\begin{align*}
0\leq\lim_{h\to\infty}\frac{\omega(h)}{h}=\omega(1)\lim_{h\to\infty}\frac{n\omega(h)}{\frac{h}{n}\omega(1)}\leq\liminf_{h\to\infty}\frac{n\omega(h)}{\omega(\frac{h}{n})},\quad n\in\N,
\end{align*}
and thus $ \lim_{h\to\infty}\omega(h)/h=0 $ by \eqref{de_condition}.

(iii): Let $ \varepsilon>0 $. Then, by \eqref{de_condition}, there exists $ n\in\N $ such that
\begin{align*}
\liminf_{h\to\infty}\frac{\omega(h)}{n\omega\big(\frac{h}{n}\big)}<\varepsilon.
\end{align*}
Further, there exists $ h_0>0 $ such that $ \eta(h)/h\leq 1/n $ for every $ h\geq h_0 $. Hence by (i), Remark \ref{re_concave} (ii) and Lemma \ref{le_lowerestimate}, we have
\begin{align*}
0<\frac{n\omega\big(\frac{h}{n}\big)}{\omega(h)}=\frac{h}{\omega(h)}\frac{\omega\big(\frac{h}{n}\big)}{\frac{h}{n}}\leq\frac{h}{\omega(h)}\frac{\omega(\eta(h))}{\eta(h)}\leq\frac{\omega_{\eta}(h)}{\omega(h)},\quad h\geq h_0,
\end{align*}
and thus
\begin{align*}
0\leq\liminf_{h\to\infty}\frac{\omega(h)}{\omega_{\eta}(h)}\leq\liminf_{h\to\infty}\frac{\omega(h)}{n\omega\big(\frac{h}{n}\big)}<\varepsilon.
\end{align*}
Hence \eqref{eq_liminf} holds.
\end{proof}

We will also need the following lemma:

\begin{lemma}\label{le_projection}
Let $ n\in\N $ and let $ G\subset\R^n $ be an open convex set such that $ 1\leq\dim(\spann(\rec(G)))<n $. Then there exists a linear surjection $ L:\R^n\to\R^2 $ such that
\begin{align}\label{eq_projection}
(1,0)\in\rec(L(G))=L(\rec(G))\subset\spann(\lbrace (1,0)\rbrace).
\end{align}
\end{lemma}

\begin{proof}
We will proceed by induction on $ n $. If $ n=2 $, then $ \dim(\spann(\rec(G)))=1 $ and thus there exists a linear bijection $ L:\R^n\to\R^2 $ such that $ (1,0)\in L(\rec(G))\subset\spann(\lbrace (1,0)\rbrace) $. Hence \eqref{eq_projection} holds by Lemma \ref{le_rock}.

Now suppose that $ n\geq 3 $ is given and the assertion holds for $ n-1 $. It follows easily from Lemma \ref{le_rec} (ii) that $ W:=\rec(G)\cap(-\rec(G)) $ is a linear subspace of $ \spann(\rec(G)) $. Thus
\begin{align*}
m:=\dim(W)\leq k:=\dim(\spann(\rec(G)))
\end{align*}
and there exist $ b_1,\dots,b_k\in\R^n $ such that $ \lbrace b_1,\dots,b_m\rbrace $ is an algebraic basis of $ W $ and $ \lbrace b_1,\dots,b_k\rbrace $ an algebraic basis of $ \spann(\rec(G)) $. We will show that there exists a linear surjection $ L_1:\R^n\to\R^{n-1} $ such that
\begin{align}\label{eq_1}
\rec(G)\cap L_1^{-1}(\lbrace 0\rbrace)\subset W,\;1\leq\dim(L_1(\spann(\rec(G))))<n-1.
\end{align}
We will distinguish three cases:

(a) $ k<n-1 $: Then it is easy to find a linear surjection $ L_1:\R^n\to\R^{n-1} $ such that $ \spann(\rec(G))\cap L_1^{-1}(\lbrace 0\rbrace)=\lbrace 0\rbrace $. Then
\begin{align*}
L_1(\spann(\rec(G)))=\spann(\lbrace L_1(b_1),\dots,L_1(b_k)\rbrace)
\end{align*}
and thus \eqref{eq_1} follows.

(b) $ m=0 $ and $ k=n-1 $: Then it is easy (cf. \cite[Lemma 5.10]{KZ} or see \cite[Consequence 2]{S}) to find $ w\in\spann\lbrace b_1,b_2\rbrace\setminus\lbrace 0\rbrace $ such that $ \rec(G) \cap\spann(\lbrace w\rbrace)=\lbrace 0\rbrace $. Then there exists a linear surjection $ L_1:\R^n\to\R^{n-1} $ such that $ L_1^{-1}(\lbrace 0\rbrace)=\spann(\lbrace w\rbrace) $. Since $ L_1(w)=0 $ and $ w\in\spann\lbrace b_1,b_2\rbrace $, it follows that $ L_1(b_1) $ and $ L_1(b_2) $ are linearly dependent and thus  
\begin{align}\label{eq_2}
L_1(\spann(\rec(G)))=\spann(\lbrace L_1(b_2),\dots,L_1(b_k)\rbrace).
\end{align}
Hence \eqref{eq_1} holds. 

(c) $ 1\leq m\leq k=n-1 $: Then we choose a linear surjection $ L_1:\R^n\to\R^{n-1} $ such that $ L_1^{-1}(\lbrace 0\rbrace)=\spann(\lbrace b_1\rbrace) $. Then $ L_1^{-1}(\lbrace 0\rbrace)\subset W $ and \eqref{eq_2} holds. Consequently \eqref{eq_1} holds.

So there exists $ L_1 $ for which \eqref{eq_1} holds. Then $ L_1(G) $ is an open convex set and
\begin{align}\label{eq_3}
L_1(\rec(G))=\rec(L_1(G))
\end{align}
by Lemma \ref{le_rock}. Thus 
\begin{align*}
\spann(\rec(L_1(G)))&=\spann(L_1(\rec(G)))=L_1(\spann(\rec(G)))
\end{align*}
and hence $ 1\leq\dim(\spann(\rec(L_1(G))))<n-1 $. Now by induction assumption there exists a linear surjection $ L_2:\R^{n-1}\to\R^2 $ such that
\begin{align}\label{eq_4}
(1,0)\in\rec(L_2(L_1(G)))=L_2(\rec(L_1(G)))\subset\spann(\lbrace (1,0)\rbrace).
\end{align}
Then $ L:=L_2\circ L_1 $ is a linear surjection and it follows from \eqref{eq_3} and \eqref{eq_4} that \eqref{eq_projection} holds. 
\end{proof}

\begin{remark}
The proof of the preceding lemma is similar to the proof of \cite[Proposition 5.13]{KZ}. Let us note that the proof of \cite[Proposition 5.13]{KZ} is incomplete. Namely, in case (b1) we forgot to consider the case when $ H $ is bounded. This can be easily corrected as follows: If $ H $ is bounded, then clearly $ \dim(\spann(\rec(G)))=1 $ and thus it is easy to find a linear surjection $ L:\R^n\to\R^{n-1} $ that satisfies \cite[(24) and (25)]{KZ}, and so we can finish the proof as in case (b2).
\end{remark}

Now we will prove the main theorem of the article. Recall that it implies Theorem \ref{th_1} by \eqref{eq_condition} and that it can also be applied to other interesting moduli than power type moduli (see Remark \ref{re_condition} (i)).

\begin{theorem}\label{th_main}
Let $ G\subset\R^n $ ($ n\geq 2 $) be an unbounded open convex set that doesn't contain a translation of a cone with non-empty interior. Let $ \omega\in\M $ be concave and satisfy \eqref{de_condition}. Then there exists $ f:G\to\R $ that is both $ \omega $-semiconvex and $ \omega $-semiconcave and $ f\notin C^{1,\omega}(G) $.
\end{theorem}

\begin{proof}
By Lemma \ref{le_rec} (iv), (v) and Lemma \ref{le_projection} there exists a linear surjection $ L:\R^n\to\R^2 $ such that
\begin{align}\label{eq_rec}
(1,0)\in\rec(L(G))=L(\rec(G))\subset\spann(\lbrace (1,0)\rbrace).
\end{align}
Then $ L(G) $ is an open convex set. We will further distinguish two cases. 

(a) $ -(1,0)\notin \rec(L(G)) $: Then it follows (see the proof of \cite[Proposition 5.13]{KZ}) from \eqref{eq_rec} that there exist $ b\in\R^2 $ and a concave $ \eta:(0,\infty)\to(0,\infty) $ such that $ \lim_{x\to\infty}\eta(x)/x=0 $ and
\begin{align}\label{eq_set}
\lbrace (x,0)\in\R^2:x>0\rbrace\subset L(G)+b\subset G_0:=\lbrace (x,y)\in\R^2:x>0,\;\vert y\vert<\eta(x)\rbrace.
\end{align}
By Proposition \ref{pr_main} and Lemma \ref{le_condition} (iii) there exists $ f_0\in C^1(G_0) $ that is both $ \omega $-semiconvex and $ \omega $-semiconcave and such that $ f'_0 $ is uniformly continuous on $ \lbrace (x,0)\in\R^2:x>0\rbrace $ with modulus $ C\omega $ for no $ C>0 $. Thus it follows from \eqref{eq_set} that $ L(G)+b\in\mathcal{G}_2(\omega) $. Then $ L(G)\in\mathcal{G}_2(\omega) $ by Remark \ref{re_gn} (iii), and hence $ G\in\mathcal{G}_n(\omega) $ by Lemma \ref{le_kz2} and \eqref{eq_rec}. Now the assertion follows from Definition \ref{de_gn}.

(b) $ -(1,0)\in\rec(L(G)) $: Then it follows from Lemma \ref{le_rec} (ii) and \eqref{eq_rec} that $ \rec(L(G))=\spann(\lbrace (1,0)\rbrace) $. Hence it is easy to find a bounded open interval $ I\subset\R $ such that $ L(G)=\R\times I $. Then we find a linear bijection $ Q:\R^2\to\R^2 $ and $ b\in\R^2 $ such that
\begin{align*}
Q(L(G))+b=\R\times(0,1).
\end{align*}
Then $ Q(L(G))+b\in\mathcal{G}_2(\omega) $ by Lemma \ref{le_kz1} and Lemma \ref{le_condition} (i), (ii). Hence $ Q(L(G))\in\mathcal{G}_2(\omega) $ by Remark \ref{re_gn} (iii). Now by Lemma \ref{le_rock} (applied to $ Q $) and \eqref{eq_rec} we have
\begin{align*}
\rec(Q(L(G)))=Q(\rec(L(G)))=Q(L(\rec(G)))
\end{align*}
and thus $ G\in\mathcal{G}_n(\omega) $ by Lemma \ref{le_kz2} (applied to $ Q\circ L $). Then the assertion follows from Definition \ref{de_gn}.
\end{proof}

\begin{acknowledgment}
I thank Luděk Zajíček for many helpful suggestions which improved this article.
\end{acknowledgment}

\end{document}